\documentclass[11pt,a4paper, reqno]{amsart}
\usepackage[T1]{fontenc}
\usepackage[utf8]{inputenc}
\usepackage[english]{babel}
\usepackage{amsmath,amssymb,amsthm,mathtools}
\usepackage[margin=25mm]{geometry}
\usepackage{microtype}
\usepackage[unicode,colorlinks=true,allcolors=blue]{hyperref}
\hypersetup{
  pdftitle={Subgroups of free groups with the same derived subgroup},
  pdfauthor={Artur Tursunbaev}
}
\makeatletter
\def\@settitle{\begin{center}\normalfont\large\bfseries\@title\end{center}}
\makeatother
\numberwithin{equation}{section}
\newtheorem{theorem}{Theorem}[section]
\newtheorem{lemma}[theorem]{Lemma}
\newtheorem{proposition}[theorem]{Proposition}
\newtheorem{corollary}[theorem]{Corollary}
\theoremstyle{definition}

\theoremstyle{remark}

\newcommand{\Z}{\mathbb Z}
\newcommand{\Q}{\mathbb Q}
\newcommand{\core}{\operatorname{core}}
\newcommand{\Norm}{\operatorname{N}}
\newcommand{\rank}{\operatorname{rank}}
\newcommand{\Roots}{\mathcal R}
\title[Subgroups with the same derived subgroup]
{Subgroups of free groups with the same derived subgroup}

\author{Artur Tursunbaev}

\address{St.~Petersburg State University,
Universitetskii pr.~28,
St.~Petersburg 198504, Russia}

\email{\href{mailto:tursunbaev.art@gmail.com}{tursunbaev.art@gmail.com}}

\date{}
\begin{document}
\begin{abstract}
We prove that any two subgroups of a free group with the same nontrivial
commutator subgroup are equal. This yields an affirmative answer to Shpilrain's Problem F15: normality of $[R,R]$ in a free group $F$ implies normality of the noncyclic subgroup $R$. No restrictions on ranks or indices are imposed. 

\end{abstract}
\maketitle

\section{Introduction}

In 1990, V.\,E.~Shpilrain posed, in the 11th edition of
\emph{The Kourovka Notebook}, the question of whether normality of
the commutator subgroup of a noncyclic subgroup of a free group implies
normality of the subgroup itself; see~\cite[Problem~11.124]{Kourovka}.
This question was subsequently restated as Problem~F15
in the book by Kapovich, Myasnikov, and Shpilrain
\cite[Problem~F15, p.~10]{KMS}.

The motivation comes from the classical Auslander--Lyndon theorem:
if $N$ and $M$ are normal subgroups of a nonabelian free group
and $N'\le M'$, then $N\le M$.
The normality assumption is essential here:
Dunwoody~\cite[p.~153]{Dun} constructed subgroups
$N,M\le F$ such that $M\triangleleft F$, $N'\le M'$, but
$N\nleq M$. Thus, the normality assumption on $N$ cannot
simply be omitted from the Auslander--Lyndon theorem.

In this context, Shpilrain's problem has the following precise formulation:
if $F$ is a noncyclic free group, $R\le F$ is
a noncyclic subgroup, and
\[
R'\triangleleft F,
\]
must we have
\[
R\triangleleft F?
\]

In this paper, we answer this question in the affirmative.
Throughout, we use the convention
\[
[a,b]=a^{-1}b^{-1}ab,
\qquad
H'=[H,H].
\]

\begin{theorem}\label{thm:main}
Let $F$ be a free group of arbitrary rank, and let $H,K\le F$. If
\[
H'=K'\ne1,
\]
then $H=K$.
\end{theorem}

\begin{corollary}[F15]\label{cor:F15}
Let $F$ be a free group, and let $R\le F$ be a noncyclic subgroup. Then
\[
R'\triangleleft F\quad\Longrightarrow\quad R\triangleleft F.
\]
\end{corollary}

\begin{proof}[Derivation of the corollary from the theorem]
The subgroup $R$ is free and noncyclic, so $R'\ne1$.
For every $f\in F$, normality of $R'$ gives
\[
(f^{-1}Rf)'=f^{-1}R'f=R'.
\]
Theorem~\ref{thm:main} implies $f^{-1}Rf=R$.
Hence $R\triangleleft F$.
\end{proof}

Another immediate consequence of the main theorem is the equality
\[
\Norm_F(H')=\Norm_F(H)
\]
for every noncyclic subgroup $H$ of a free group $F$.

The main theorem can also be formulated as the assertion
that the map
\[
H\longmapsto H'
\]
is injective on the set of subgroups $H\le F$ with $H'\ne1$.

\section{Classical results and the finite-index case}

\begin{proposition}[Auslander--Lyndon]\label{prop:AL}
Let $E$ be a noncyclic free group.
\begin{itemize}
\item[(a)] If $1\ne N\triangleleft E$, then the action of $E/N$ on $N/N'$ by conjugation is faithful. In other words, if $g\in E$ and $[N,g]\le N'$, then $g\in N$.
\item[(b)] If $N_1,N_2\triangleleft E$ and $N_1'\le N_2'$, then $N_1\le N_2$. In particular, $N_1'=N_2'$ implies $N_1=N_2$.
\end{itemize}
\end{proposition}

These are Theorems~1 and~3 of \cite[pp.~929 and~931]{AL}, respectively. Neither statement requires the rank to be finite.

\begin{lemma}\label{lem:finiteindex}
Let $E$ be a noncyclic free group, and let $H\le E$ satisfy $[E:H]<\infty$ and $D=H'\triangleleft E$. Then $H\triangleleft E$.
\end{lemma}

\begin{proof}
Set
\[
C=\core_E(H)=\bigcap_{g\in E}g^{-1}Hg,
\qquad \Omega=E/C,\qquad P=H/C.
\]
The group $\Omega$ is finite. Normality of $D$ and the inclusion $D\le H$ imply $D\le C$. Hence $C'\le D\le C$.

All modules in this proof are right modules. Consider
\[
V=(C/C')\otimes\Q,
\qquad W=(D/C')\otimes\Q\le V.
\]
Conjugation defines a right $\Q\Omega$-module structure on $V$, and normality of $D$ makes $W$ a submodule. For a subgroup $B\le\Omega$, let $I_B$ denote the augmentation ideal of $\Q B$.

The five-term exact sequence in homology \cite[p.~47]{Brown} for the extension
$C\triangleleft H$ yields an exact sequence
\[
H_2(P;\Q)\longrightarrow V_P\longrightarrow H_1(H;\Q)
\longrightarrow H_1(P;\Q)\longrightarrow0.
\]
Since $P$ is finite, both outer homology terms vanish, giving an isomorphism
\[
V_P=V/VI_P\ \cong\ H_1(H;\Q).
\]
The kernel of the homomorphism $C/C'\to H/H'$ is
\[
(C\cap H')/C'=D/C'.
\]
Since $\Q$ is a flat $\Z$-module, the kernel of the natural
map $V\to H_1(H;\Q)$ is $W$. On the other hand,
by the preceding isomorphism, this kernel is $VI_P$. Therefore,
\begin{equation}\label{eq:W}
W=VI_P.
\end{equation}

We show that $V$ contains a direct $\Q\Omega$-summand
isomorphic to $\Q\Omega$. Indeed, if $\rank E=n<\infty$, the relation module exact sequence for $C\triangleleft E$ has the form
\[
0\longrightarrow V\longrightarrow(\Q\Omega)^n\longrightarrow I_\Omega\longrightarrow0.
\]
A more general form of this sequence can be found, for example,
in Blackburn~\cite[p.~470]{Blackburn}.

The semisimplicity of $\Q\Omega$ and the decomposition $\Q\Omega\cong\Q\oplus I_\Omega$ give
\[
V\cong\Q\oplus(\Q\Omega)^{n-1}.
\]
Here $n\ge2$. If $\rank E=\infty$, choose a finite subset of a free
basis whose images generate $\Omega$, and one further basis
element $x$. Replacing $x$ by $xw^{-1}$ for a suitable word $w$
in the chosen finite subset yields a free basis
$X$ of $E$ containing an element $x\in C$.

For this basis, the relation module exact sequence has the form
\[
0\longrightarrow V\longrightarrow
\bigoplus_{y\in X}e_y\Q\Omega
\xrightarrow{\partial}I_\Omega\longrightarrow0,
\qquad
\partial(e_y)=\bar y-1.
\]
Since $x\in C$, we have $\partial(e_x)=0$, and hence
\[
V=e_x\Q\Omega\oplus
\left(
V\cap\bigoplus_{y\in X\setminus\{x\}}e_y\Q\Omega
\right).
\]
Thus, in this case as well, $V$ contains a direct
$\Q\Omega$-summand isomorphic to $\Q\Omega$.

In either case, choose a direct summand
$U\cong\Q\Omega$ and write
\[
V=U\oplus U_0.
\]
Let $\pi:V\to U$ denote the corresponding
$\Q\Omega$-linear projection. By~\eqref{eq:W}, we obtain
\[
\pi(W)=\pi(VI_P)=\pi(V)I_P=UI_P.
\]
Since $W$ is a $\Q\Omega$-submodule, its image $UI_P$
is stable under right multiplication by elements of $\Omega$. Under
the identification $U\cong\Q\Omega$, it follows that
$J=\Q\Omega I_P$ is also stable under right multiplication by
elements of $\Omega$.

As a $\Q$-vector space, $\Q\Omega/J$ has a basis
indexed by the left cosets $\omega P$.
Indeed, $J$ is spanned by all differences
$\omega p-\omega$, where $\omega\in\Omega$ and $p\in P$.
For $p\in P$ and $\omega\in\Omega$, stability of $J$ under
right multiplication by elements of $\Omega$ gives
\[
(p-1)\omega\in J,
\]
whence
\[
p\omega P=\omega P,
\qquad
\omega^{-1}p\omega\in P.
\]
Thus $P\triangleleft\Omega$, and hence $H\triangleleft E$.
\end{proof}

\section{Chains in Schreier graphs}

Fix a free basis $X$ of the free group $E$. For $H\le E$,
let $\Gamma_H$ denote the Schreier graph with vertices $Hg$ and
positively oriented edges
\[
Hg\xrightarrow{x}Hgx,
\qquad x\in X.
\]
For an oriented edge $e$, let $o(e)$ and $t(e)$ denote
its initial and terminal vertices, respectively.

Every word $w$ over the alphabet $X^{\pm1}$ determines a unique path
starting at each vertex of $\Gamma_H$. Reading a fixed word defines
a permutation of the vertex set of $\Gamma_H$.

Let $w\in H$. The corresponding path starting at the vertex $H$
is closed. Let
\[
\sigma_H(w)\in C_1(\Gamma_H;\mathbb Z)
\]
denote its $1$-chain, where $C_1(\Gamma_H;\mathbb Z)$ is the free abelian
group on the positively oriented edges of $\Gamma_H$.
The coefficient of an edge in $\sigma_H(w)$ is the number of times it is traversed
in the positive direction minus the number of times it is traversed
in the opposite direction. Free reduction of the word
does not change this chain.

For $u,v\in H$, we have
\[
\sigma_H(uv)=\sigma_H(u)+\sigma_H(v),
\]
so this defines a homomorphism
\[
\sigma_H:H\longrightarrow C_1(\Gamma_H;\mathbb Z).
\]

The standard identification
\[
\pi_1(\Gamma_H,H)\cong H
\]
and passage to the abelianization give
\begin{equation}\label{eq:sigma-kernel}
\ker\sigma_H=H'.
\end{equation}
Indeed, since $\Gamma_H$ is a graph,
\[
H_1(\Gamma_H;\mathbb Z)=Z_1(\Gamma_H;\mathbb Z)
\subseteq C_1(\Gamma_H;\mathbb Z),
\]
a closed path has zero homology class if and
only if its associated $1$-chain is zero.

\begin{lemma}[On a central coset]\label{lem:center}
Let $H\le E$, $D=H'\triangleleft E$, and suppose that there exists $c\in H\setminus D$ such that
\[
cD\in Z(E/D).
\]
Then $[E:H]<\infty$. More precisely, for any word of length $\ell$ representing $c$ with respect to the basis $X$, we have $[E:H]\le\ell$.
\end{lemma}

\begin{proof}
For every $g\in E$, we have
\[
gcg^{-1}c^{-1}\in D=H'.
\]
Thus $gcg^{-1}\in H$, and the path corresponding to the word $c$
and starting at the vertex $Hg$ is closed. The path in $\Gamma_H$
corresponding to the word $gcg^{-1}$ and starting at $H$ has the form
\[
g\cdot c\cdot g^{-1},
\]
where the chains of the segments $g$ and $g^{-1}$ cancel each other.

Moreover, $gcg^{-1}$ and $c$ have the same image in $H/H'$.
Consequently, the chain of the closed path corresponding to the word $c$
and starting at an arbitrary vertex $Hg$ is
\[
\alpha=\sigma_H(c)\ne0.
\]
The fact that $\alpha\ne0$ follows from $c\notin H'$ and
\eqref{eq:sigma-kernel}.

Fix a positively oriented edge $e$ whose coefficient
in $\alpha$ is nonzero. Since the chain of every closed
path corresponding to the word $c$ is $\alpha$, each such path traverses
$e$ in one of the two directions.

Write
\[
c=s_1\cdots s_\ell,\qquad s_j\in X^{\pm1},
\]
and set $u_{j-1}=s_1\cdots s_{j-1}$. Let the edge $e$ be
labeled by $x\in X$. If the $j$-th step of the path starting at a vertex $v$ traverses $e$
in the positive direction, then $s_j=x$ and
\[
v u_{j-1}=o(e).
\]
If $e$ is traversed in the opposite direction, then
$s_j=x^{-1}$ and
\[
v u_{j-1}=t(e).
\]
Consequently,
\[
V(\Gamma_H)\subseteq
\{\,o(e)u_{j-1}^{-1}:1\le j\le\ell,\ s_j=x\,\}
\cup
\{\,t(e)u_{j-1}^{-1}:1\le j\le\ell,\ s_j=x^{-1}\,\}.
\]
Thus, each $j$ corresponds to at most one initial vertex,
and hence
\[
[E:H]=|V(\Gamma_H)|\le \ell.
\]
\end{proof}

\section{Intersection dichotomy}

\begin{lemma}\label{lem:intersection}
Let $H,K\le F$ and $H'=K'=D\ne1$. Then
\[
H=K \qquad\text{or}\qquad H\cap K=D.
\]
\end{lemma}

\begin{proof}
The subgroup $D$ is normal in both $H$ and $K$, so
\[
D\triangleleft E:=\langle H,K\rangle.
\]

Suppose that there exists
$c\in(H\cap K)\setminus D$.
The subgroups $H/D$ and $K/D$ are abelian and generate $E/D$.
Since $cD$ belongs to both, it centralizes $H/D$ and $K/D$,
and hence
\[
cD\in Z(E/D).
\]

Lemma~\ref{lem:center}, applied to $H\le E$ and $K\le E$,
gives
\[
[E:H]<\infty,\qquad [E:K]<\infty.
\]
By Lemma~\ref{lem:finiteindex}, both subgroups are normal in $E$.
Then Proposition~\ref{prop:AL}(b), applied to
$H,K\triangleleft E$, gives $H=K$.

\end{proof}

\section{Finiteness of the family of subgroups with a common commutator subgroup}

\begin{lemma}\label{lem:avoidance}
Let $D\le F$ and $d\in D\setminus D'$. Then there exists a finite set $T\subset F\setminus D$ with the following property:
\begin{equation}\label{eq:avoidance}
D\le H\le F,\quad H\cap T=\varnothing
\quad\Longrightarrow\quad d\notin H'.
\end{equation}
\end{lemma}

\begin{proof}
Fix a free basis $X$ of $F$ and choose a word
representing $d$. The corresponding closed path in $\Gamma_D$
starting at the vertex $D$ has a nonzero $1$-chain
$\sigma_D(d)$ by \eqref{eq:sigma-kernel}.

Let
\[
Db_1,\ldots,Db_m
\]
be all the distinct vertices of this path; for each of them, choose
a representative $b_i$ among the prefixes of the chosen word. Set
\[
T=\{\,b_i b_j^{-1}:1\le i,j\le m,\ i\ne j\,\}.
\]
Since the cosets $Db_1,\ldots,Db_m$ are pairwise distinct, we have
\[
T\cap D=\varnothing.
\]

Let $D\le H$ and $H\cap T=\varnothing$. The natural map
\[
p:\Gamma_D\longrightarrow\Gamma_H,
\qquad
Dg\longmapsto Hg,
\]
is injective on the vertices of the chosen path: the equality
$Hb_i=Hb_j$ would imply $b_i b_j^{-1}\in H$, contradicting
$H\cap T=\varnothing$.

Let
\[
p_\#:C_1(\Gamma_D;\mathbb Z)\longrightarrow C_1(\Gamma_H;\mathbb Z)
\]
denote the induced homomorphism on $1$-chains.

The map $p$ is also injective on the positively oriented edges
occurring in the chosen path. Indeed, both endpoints of each such
edge belong to the chosen set of vertices, and a positively
oriented edge of a Schreier graph is uniquely determined by its
initial vertex and label. Since $p$ preserves labels and orientations,
distinct edges in the support of $\sigma_D(d)$ have distinct images. Therefore,
\[
\sigma_H(d)=p_\#\sigma_D(d)\ne0.
\]
By \eqref{eq:sigma-kernel}, we obtain $d\notin H'$.
\end{proof}

\begin{lemma}\label{lem:finitefamily}
For every nontrivial subgroup $D\le F$, the set
\[
\Roots(D)=\{H\le F:H'=D\}
\]
is finite.
\end{lemma}

\begin{proof}
By the Nielsen--Schreier theorem~\cite[\S2.4]{MKS}, the subgroup $D$ is free;
since $D\ne1$, we have $D/D'\ne1$. Choose $d\in D\setminus D'$ and a finite set
$T$ as in Lemma~\ref{lem:avoidance}.

If $H\in\Roots(D)$, then $D=H'\le H$ and $d\in H'$.
By \eqref{eq:avoidance}, we have $H\cap T\ne\varnothing$.

If $H,K\in\Roots(D)$ and $H\ne K$, then Lemma~\ref{lem:intersection}
gives $H\cap K=D$. Since $T\cap D=\varnothing$, the sets
$H\cap T$, $H\in\Roots(D)$, are pairwise disjoint. Hence
\[
|\Roots(D)|\le |T|<\infty.
\]
\end{proof}

\section{Proof of the main theorem}

\begin{proof}[Proof of Theorem~\ref{thm:main}]
Let $D=H'=K'\ne1$. By Lemma~\ref{lem:finitefamily}, the family
$\Roots(D)$ is nonempty and finite. Set
\[
E=\Norm_F(D).
\]
Every subgroup $S\in\Roots(D)$ is contained in $E$, since
$S$ normalizes its commutator subgroup $S'=D$. The group $E$ acts
by conjugation on $\Roots(D)$: for $g\in E$, we have
\[
(g^{-1}Sg)'=g^{-1}Dg=D.
\]
Let
\[
L=\bigcap_{S\in\Roots(D)}\Norm_E(S)
\]
be the kernel of this action. Then $L\triangleleft E$ and $[E:L]<\infty$.
Set $m=[E:L]$. By Lagrange's theorem,
\begin{equation}\label{eq:powers-in-L}
g^m\in L
\qquad (g\in E).
\end{equation}

Suppose that $H\ne K$. Then Lemma~\ref{lem:intersection} gives
\begin{equation}\label{eq:HKintersection}
H\cap K=D.
\end{equation}
Choose $k\in K\setminus D$ and set $q=k^m$. Since
$K/D=K/K'$ is a free abelian group, $q\notin D$. By
\eqref{eq:powers-in-L}, we have
\begin{equation}\label{eq:q-in-L}
q\in K\cap L,\qquad q\notin D.
\end{equation}
For $h\in H$, set $p=h^m$. Then $p\in H\cap L$.
Since $q\in L$, the element $q$ normalizes $H$, and hence
$[p,q]\in H$. On the other hand, $p\in L$ normalizes $K$, and
$q\in K$, so $[p,q]\in K$. By
\eqref{eq:HKintersection}, we obtain
\[
[h^m,q]=[p,q]\in D.
\]

Conjugation by $q$ induces an automorphism $\alpha$ of the group
\[
A=H/D=H/H'.
\]
In additive notation, the relation $[h^m,q]\in D$ means that
\[
m\alpha(a)=\alpha(ma)=ma
\qquad (a\in A).
\]
Since $A$ is a free abelian group, it is torsion-free.
Thus $\alpha(a)=a$ for all $a\in A$, that is,
\begin{equation}\label{eq:trivial-action}
[H,q]\le D=H'.
\end{equation}

Consider
\[
E_0=\langle H,q\rangle.
\]
The group $E_0$ is free and noncyclic, since $H'\ne1$.
As $q\in L$, we have $H\triangleleft E_0$. Therefore,
Proposition~\ref{prop:AL}(a), applied to $H\triangleleft E_0$,
together with \eqref{eq:trivial-action}, gives $q\in H$.
But $q\in K$, so \eqref{eq:HKintersection} yields
$q\in D$, contradicting \eqref{eq:q-in-L}. Hence
$H=K$.
\end{proof}

\begin{corollary}[Normalizers]
For every noncyclic subgroup $H$ of a free group $F$, we have
\[
\Norm_F(H')=\Norm_F(H).
\]
\end{corollary}

\begin{proof}
The inclusion $\Norm_F(H)\le \Norm_F(H')$ is clear: if $g$ normalizes $H$, then
\[
g^{-1}H'g=(g^{-1}Hg)'=H'.
\]
If $g\in \Norm_F(H')$, then
\[
(g^{-1}Hg)'=H'\ne1,
\]
and Theorem~\ref{thm:main} gives $g^{-1}Hg=H$.
\end{proof}

\end{document}